\documentclass{amsart}

\newtheorem{theorem}{Theorem}[section]
\newtheorem{proposition}[theorem]{Proposition}
\newtheorem{corollary}[theorem]{Corollary}
\newtheorem{lemma}[theorem]{Lemma}

\theoremstyle{definition}

\theoremstyle{remark}
\newtheorem{remark}[theorem]{Remark}

\numberwithin{equation}{section}

\DeclareMathOperator{\Spec}{Spec}

\DeclareMathOperator{\Proj}{Proj}

\DeclareMathOperator{\Res}{Res}
\DeclareMathOperator{\D}{D}

\begin{document}

\title{Further evaluation of Wahl vanishing theorems for surface singularities
in characteristic $p$}

\author{Masayuki Hirokado}
\address{Graduate school of Information Sciences, Hiroshima City University, Ozuka-higashi, Asaminami-ku, Hiroshima 731-3194, Japan}
\email{hirokado@math.info.hiroshima-cu.ac.jp}


\subjclass[2010]{ Primary 14B05, 14G17; 
Secondary 14B07, 14J17.}



\keywords{Rational double points \and Wahl's vanishing theorems \and local cohomologies
\and equivariant resolutions \and equisingular deformations.
}

\begin{abstract}
Let $(\Spec R, \frak m)$ be a rational double point
defined over an algebraically closed field $k$
of characteristic $p\geq 0$.
We evaluate further the dimensions of the local cohomology groups
which were treated by Wahl in 1975 as vanishing theorem C (resp. D)
under the assumption that $p$
is a very good prime
(resp. good prime) with respect to $(\Spec R, \frak m)$. 
We use Artin's classification of rational double points
and completely determine the
dimensions $\dim_k H_E^1(S_X)$,
$\dim_k H_E^1(S_X\otimes \mathcal O_X(E))$,
supplementing Wahl's theorems.
In the proof we construct derivations concretely  
which do not lift to the minimal resolution
$X\to \Spec R$, as well as non-trivial equisingular families 
which inject into a versal deformation of the rational double point
$(\Spec R, \mathfrak m)$.
\end{abstract}

\maketitle

\markboth{MASAYUKI HIROKADO}{Wahl's vanishing theorems in characteristic $p$}


.

\section{Introduction}\label{sec:1}
In 1975 Jonathan Wahl proved the Grauert-Riemenschneider vanighing theorem
along with three other types of vanishing theorems for a surface singularity
$(\Spec R, \mathfrak m)$
defined over an algebraically closed field $k$ of characteristic 
$p\geq 0$~\cite{Wahl75}.
Last decades witnessed that these theorems on local 
cohomology groups have played influential roles 
in the theory  of surface singularities.
Among them are what he calls Theorems C, D which bear 
restrictions on the characteristic of the ground field~$k$.

\medskip

\noindent
{\bf Theorem C} {\it Let $X\to \Spec R$ be the minimal resolution of a rational double point.
Then $H_E^1(S_X)=0$, and in particular the resolution is equivariant, except the 
following cases:  }
$$\begin{array}{ll}
A_n\quad &p\,|\,n+1,\phantom{11111111111111111111111111111111}\\
D_n&p=2,\\
E_6&p=2, 3,\\
E_7&p=2,3,\\
E_8&p=2,3,5.\\
\end{array}
$$

\noindent
{\bf Theorem D} {\it Let $X\to \Spec R$ be the minimal resolution of a rational double point.
Then $H_E^1(S_X(E))=0$, except in the following cases:  }
$$\begin{array}{ll}
D_n\quad &p=2,\phantom{11111111111111111111111111111111}\\
E_6&p=2, 3,\\
E_7&p=2,3,\\
E_8&p=2,3,5.\\
\end{array}
$$
In these theorems $E=\cup_{i} E_i$ 
denotes the exceptional divisor in $X$
and its irreducible decomposition, $S_X$ is  
the sheaf of logarithmic derivations $S_X=\Theta_{X}(-\log E)$ (cf. \cite{Wahl76})
which fits in an exact sequence:
$$
0\to S_X\to \Theta_{X}\to \bigoplus_i \mathcal N_{E_i/X}\to 0.
$$

In this article we completely evaluate the dimensions of 
$k$-vector spaces
$H_E^1(S_X)$, $H_E^1(S_X\otimes \mathcal O_X(E))$ 
for each isomorphism class $(\Spec R, \mathfrak m)$
of rational double points. 

Before we proceed, there may be two things one should be aware of.  
One is that the classification of rational double points over an 
algebraically closed field $k$ of arbitrary characteristic 
was completed 
by Artin \cite{Artin77} in 1977.
Wahl's original vanishing theorems do not depend on it.
Second, in the theory of root systems
there are notions of good primes~\cite[Ch. I, \S 4]{Springer-Steinberg70}, 
very good primes~\cite[3.13]{Slodowy80},
which coincide beautifully with Wahl's theorems.
It follows that for irreducible root systems, the bad(=\,not good) 
prime numbers are:  
$$\begin{array}{ll}
B_n, C_n, D_n&p=2,\phantom{11111111111111111111111111111111}\\
E_6, E_7, F_4, G_2 \quad &p=2,3,\\
E_8&p=2,3,5.\\
\end{array}
$$
For root system $A_n$, $p$ is defined to be a very good prime, if $p$ does not devide $n+1$. 
This suggests that there are further relationships yet
to be discovered between rational surface singularities and 
Lie algebras.

Our main theorems are:

\noindent
{\bf Theorem~\ref{theorem:main01}}\ {\it
Let $X\to \Spec R$ be the minimal resolution of a rational double point
defined over an algebraically closed field $k$ of characteristic $p\ne 2$.
Then the following assertions hold.
\begin{itemize}
\item [{\rm i)}] The natural morphism $H^1(S_X)\to 
H^1(X\setminus E, S_X)$
is an inclusion.
\item [{\rm ii)}] The dimension of $H_E^1(S_X\otimes \mathcal O_X(E))$
is zero except:
$$
\begin{cases}
1 \quad &\text{for}\quad E_8^0\text{ in }p=5$ and $E_6^0, E_7^0, E_8^1 \text{ in }p=3, \\
2  \quad &\text{for}\quad E_8^0\text{ in }p=3. \\
\end{cases}
$$ 
\item [{\rm iii)}]  One has an isomorphism 
$$H^0(X\setminus E, \Theta_{X})/H^0(X, \Theta_{X})\cong H_E^1(S_X),
$$
whose dimension is zero with the following exceptions:
$$
\begin{cases}
1   \quad & \text{for}\quad A_n  \text{ with }p\,|\,(n+1),\\
1  &\text{for}\quad E_8^0\text{ in }p=5 \text{ and } E_6^1, E_7^0, E_8^1\text{ in }p=3, \\
2  &\text{for}\quad E_6^0, E_8^0\text{ in }p=3. \\
\end{cases}
$$ 
\end{itemize}
}

\noindent
{\bf Theorem~\ref{theorem:main02}}\ {\it
Let $X\to \Spec R$ be the minimal resolution of a rational double point
defined over an algebraically closed field $k$ of characteristic $2$.
Then we have the following assertions.
\begin{itemize}
\item [{\rm i)}] The natural morphism $H^1(S_X)\to 
H^1(X\setminus E, S_X)$
is an inclusion.
\item [{\rm ii)}] The dimension of $H_E^1(S_X\otimes \mathcal O_X(E))$
is zero except:
$$
\begin{cases}
1 \quad &\text{for}\quad E_6^0, E_7^2, E_8^3, \\
2 \quad &\text{for}\quad E_7^1, E_8^2, \\
3 \quad &\text{for}\quad E_7^0, E_8^1, \\
4 \quad &\text{for}\quad E_8^0, \\
n-1-r \quad &\text{for}\quad D_{2n}^r, D_{2n+1}^r.\\
\end{cases}
$$ 
\item [{\rm iii)}]  One has an isomorphism 
$$H^0(X\setminus E, \Theta_{X})/H^0(X, \Theta_{X})\cong H_E^1(S_X),
$$
whose dimension is zero with the following exceptions:
$$
\begin{cases}
1   \quad & \text{for}\quad A_n  \text{ with }2\,|\,(n+1),\\
1  \quad&\text{for}\quad E_6^0, E_7^3, E_8^3, \\
2  \quad&\text{for}\quad E_7^2, E_8^2, \\
3  \quad&\text{for}\quad E_7^1, E_8^1, \\
4  \quad&\text{for}\quad E_7^0, E_8^0, \\
n+1-r  \quad&\text{for}\quad D_{2n}^r, \\
n-r  \quad&\text{for}\quad D_{2n+1}^r. \\
\end{cases}
$$ 
\end{itemize}
}

\begin{corollary}
For a rational double point $(\Spec R, \frak m)$ of the following type, 
one has $H_E^1(S_X)=0$, 
in particular, the minimal resolution $X\to \Spec R$ is equivariant. 
$$
E_8^1\text{ in }p=5, \quad E_7^1, E_8^2 \text{ in }p=3, \quad E_6^1, E_8^4 \text{ in }p=2.
$$
\end{corollary}

The assertion i) in both theorems gives an answer 
to the question asked in \cite[(5.18.2)]{Wahl75} for rational double points. 

\bigskip
Wahl's vanishing theorems have come into focus
in the study of three dimensional 
canonical singularities in arbitrary characteristic 
and we needed further evaluation of dimensions of $H_E^1(S_X)$, 
$H_E^1(S_X\otimes \mathcal O_X(E))$.
One has the smooth morphism from the simultaneous resolution functor
of a versal deformation of a rational double point $(\Spec R, \mathfrak m)$
\cite{Artin74} 
to the deformation functor of its minimal resolution $X$: 
$$
\Res \mathcal X/S \to \D_X.
$$
Non-zero elements of $H_E^1(S_X)$ correspond to families
of resolutions which are
non-trivial  in $\Res \mathcal X/S $,
but  map to trivial deformations of $X$.
This is linked with 
phenomena of three dimensional canonical singularities 
peculiar to positive characteristic $p$
which are recently observed in \cite{Hirokado16}, \cite{HIS13}, \cite{IR16},
\cite{ST17}.

\section{Preliminaries}\label{sec:2}

When we say $(\Spec R, \frak m)$ is a surface singularity, 
it is understood that $(R, \frak m)$ is a two dimensional excellent 
normal local ring with the maximal ideal $\frak m$.

We use the term {\it equisingular deformations} of the resolution in the sense 
of Wahl~\cite{Wahl76}.

The following proposition is in implicit form in Artin's 
work \cite[Corollary~4.6]{Artin74}.
Here $\mathfrak R$ denotes a locally quasi-separated algebraic space which 
represents the functor $\Res \mathcal X/S$ of simultaneous resolutions 
of families of a surface singularity $(\Spec R, \mathfrak m)$.

\begin{proposition}
Suppose $\mathcal X/S$ is a versal deformation of a rational 
surface singularity $(\Spec R, \frak m)$ at $s_0\in S$ and has minimum tangent
space dimension there.
Then the universal family $\mathcal X_{\mathfrak R}'/\mathfrak R$ is a versal deformation of $X$ with minimum tangent
space dimension, if and only if 
the minimal resolution $X\to \Spec R$ is equivariant. 
\end{proposition}

\begin{proof} $(\Rightarrow)$
Suppose that 
there exists a non-zero element   
$\theta\in H^0(X\setminus E, \Theta_{X})/H^0(X, \Theta_{X})$.
Then one has a diagram 
$$
\begin{array}{ccc}
X\times T & & X\times T\\
\phantom{11} \downarrow_{\pi \times \mathrm{id}} & &\phantom{11} \downarrow_{\pi \times \mathrm{id}} \\
\Spec R\times T &\stackrel{\varphi}\longrightarrow &\Spec R\times T, \\
\end{array}
$$  
where $T:=\Spec k[\epsilon]/(\epsilon^2)$ and 
$\varphi$ is an isomorphism given by sending 
$f\in R$ to $f+\theta(f)\epsilon$. 
But no morphism $X\times T \to X\times T$
makes this diagram commutative.  
One may consider the map $\varphi\circ (\pi\times \mathrm{id})$
as a resolution of $\Spec R\times T$.
This gives a nontrivial extension of $T\to S$ to $T\to \mathfrak R$,
although $X\times T$ is a trivial deformation of $X$.
This contradicts the fact 
that $\mathcal X_{\mathfrak R}'/\mathfrak R$ has the minimum tangent space dimension as a 
versal deformation of $X$.

$(\Leftarrow)$ This is proved by Artin. 
\end{proof}

The following is a refinement of the inequality which  
Shepherd-Barron used in \cite[Proposition 3.1]{Shepherd-Barron01}.

\begin{proposition}
For a rational double or triple point $(\Spec R, \frak m)$, 
one has the inequality
$$
\dim_k H^1(\Theta_{X})+\dim_k H^0(X\setminus E, \Theta_{X})/H^0(X, \Theta_{X}) \leq \tau_0,
$$
where $\tau_0$ is the Tjurina number of $(\Spec R, \frak m)$.
\end{proposition}

\begin{proof}
This is based on the fact that there is a smooth morphism of functors
$\Res\mathcal X/S \to \D_{X}$~\cite[Lemma 3.3]{Artin74},
and the tangent space  
$\Res\mathcal X/S(k[\epsilon]/(\epsilon^2))$ 
has its dimension $\tau_0$ \cite[Theorem 3]{Artin74}. 
The inequality is the dimension formula 
for the surjective $k$-linear mapping
$\Res\mathcal X/S(k[\epsilon]/(\epsilon^2)) \to \D_{X}(k[\epsilon]/(\epsilon^2))$.
\end{proof}

\begin{proposition}
Let $(\Spec R, \frak m)$ be a rational surface singularity defined over an
algebraically closed field $k$ of characteristic $p\geq 0$, 
and $\pi :X\to \Spec R$ be its minimal resolution
with the reduced exceptional divisor 
$\displaystyle E=\cup_i E_i$.
Then we have the equalities
\begin{eqnarray*}
\dim_k H_E^1(S_X)=\dim_k H^1(S_X\otimes \mathcal O_{X}(E+2K_{X})), \\
\dim_k H_E^1(S_X\otimes \mathcal O_{X}(E+2K_{X}))=\dim_k H^1(S_X),
\end{eqnarray*}
where $S_X$ is a locally free sheaf defined as the kernel 
of the surjection from the tangent to normal sheaves
$\displaystyle \Theta_{X}\to \bigoplus_i \mathcal N_{E_i/X}$. 
In particular, the local cohomology groups $H_E^1(S_X)$, 
$H_E^1(S_X\otimes \mathcal O_{X}(E+2K_{X}))$
are finite dimensional $k$-vector spaces.
\end{proposition}

\begin{proof}
We use the exact sequence given by Wahl~\cite[(1.2)]{Wahl85}
$$
0\to \Omega_{X} \to S_X^{\vee} \to \bigoplus_i \mathcal O_{E_i} \to 0,
$$
from which follows 
$\wedge^2(S_X^{\vee})\cong \mathcal O_{X}(E+K_{X})$ 
hence $S_X^{\vee}\cong S_X\otimes \mathcal O_{X}(E+K_{X})$.
We combine this with the Grothendieck local duality theorem, 
to have the equalities
$\dim_k H_E^1(S_X)=\dim_k H^1(S_X^{\vee}\otimes K_{X})
=\dim_k H^1(S_X\otimes \mathcal O_{X}(E+2K_{X}))$ as well as
$\dim_k H_E^1(S_X\otimes \mathcal O_{X}(E+2K_{X}))=
\dim_k H^1(S_X^{\vee}\otimes \mathcal O_{X}(-E-K_{X}))
=\dim_k H^1(S_X)$.
Then we use the Leray spectral sequence
$E_1^{i,j}:=R^i\eta_*R^j\pi_*\mathcal F \Rightarrow R^{i+j}(\eta\circ\pi)_*\mathcal F$,
where $\eta: \Spec R\to \Spec k$ is the structual morphism and
$\mathcal F\cong S_X$ or 
$S_X\otimes \mathcal O_{X}(E+2K_{X})$.
Then we find the $k$-vector spaces in question are of finite dimension.
This indeed follows from the exact sequence 
$0\to H^1(\pi_*\mathcal F)\to H^1(\mathcal F)\to H^0(R^1\pi_*\mathcal F)$
with the first term zero, 
and the last term of finite dimension. 
\end{proof}

The following theorem was pointed out by Liedtke and Satriano 
\cite[Proposition~4.6]{Liedtke-Satriano14}.

\begin{theorem}
For a rational double point $(\Spec R, \frak m)$, we have the equality
$$
\dim_k H^1(\Theta_{X})=\dim_k  H_E^1(\Theta_{X})
=\#\{-2 \mbox{ curves in }E\}+h^1(S_X),
$$
where $\pi:X\to \Spec R$ is the minimal resolution with 
the exceptional divisor $E$.
\end{theorem}

\begin{proof}
We repeat the argument in \cite[Theorem 6.1]{Wahl75}.
One has
$$
0\to \Theta_{X}\to S_X(E) \to \Theta_E\otimes \mathcal N_{E/X}\to 0,
$$
which gives
$$
0 \to H_E^0( \Theta_E\otimes \mathcal N_{E/X})\to H_E^1(\Theta_{X})\to
H_E^1(S_X(E))\to 0.
$$
The local duality theorem gives the second equality. 
For the first equality we use the local duality theorem and the 
standard isomorphism $\wedge^2 \Omega_X\cong K_X$.
\end{proof}

\section{Lower estimates of 
$\dim_k H^1(S_X)$, $\dim_k H_E^1(S_X)$}\label{sec:3}

Observing Artin's list of rational double points~\cite{Artin74} 
enables one to get the lower bounds of dimensions of 
$H^1(S_X)$ and $H_E^1(S_{X})$.

\begin{proposition}\label{prop:equisingular}
Let $X\to \Spec R$ be the minimal resolution of a rational double point
$(\Spec R, \mathfrak m)$.
If $(\Spec R, \mathfrak m)$ is of the following type,
then one has 
non-trivial equisingular deformations which 
provide
the lower bound of $\dim_kH^1(S_X)$ as
$$
\begin{cases}
1 \quad &\text{for}\quad E_8^0\text{ in }p=5$ and \ $E_8^1, \ E_6^0, \ E_7^0\text{ in }p=3, \\
2  \quad &\text{for}\quad E_8^0\text{ in }p=3. \\
\end{cases}
$$ 
Moreover, the subspace these equisingular families generate does not 
collapse in the tangent map
$$H^1(S_X)\to H^1(X\setminus E,S_X).
$$ 

\end{proposition}

\begin{proof}
We give concrete one parameter deformations 
$\mathcal X\to \Spec k[s]$ and construct simultaneous
resolutions by blowing up the singular loci.

For $E_8^0\text{ in }p=5$, we have a deformation
$z^2+x^3+y^5+sxy^4=0$.
This has a simultaneous resolution with no base extension,
and has $E_8^0$ singularity on the special fiber $\mathcal X_0$, and 
$E_8^1$ singularity on a general fiber $\mathcal X_s$ ($s\ne 0$).

Similarly we present non-trivial one-parameter deformations 
which admit simultaneous resolutions with no base extension.
These result in equisingular deformations.
For $E_7^0\text{ in }p=3$, $z^2+x^3+xy^3+sx^2y^2=0$.
For $E_6^0\text{ in }p=3$, $z^2+x^3+y^4+sx^2y^2=0$.
For $E_8^1\text{ in }p=3$, $z^2+x^3+y^5+x^2y^3+sx^2y^2=0$.
For $E_8^0\text{ in }p=3$, $z^2+x^3+y^5+s_1x^2y^3=0$ and
$z^2+x^3+y^5+s_2x^2y^2=0$.

Each family
is induced by an injection to a versal deformation of the 
rational doble point $(\Spec R, \mathfrak m)$,
and one gets the last assertion.  
\end{proof}

\begin{proposition}
Let $X\to \Spec R$ be the minimal resolution of a rational double point
$(\Spec R, \mathfrak m)$ of the following type.
Then one calculates the dimension of $H^0(X\setminus E, S_X)/H^0(X, S_X)$ 
as  

$$
\begin{cases}
1 \quad  & \text{for}\quad A_n  \text{ with }p\,|\,(n+1),\\
1  &\text{for}\quad E_8^0\text{ in }p=5 \text{ and } E_7^0\text{ in }p=3, \\
2  &\text{for}\quad E_6^0, E_8^0\text{ in }p=3. \\
\end{cases}
$$ 
For $E_8^1$, $E_6^1$ in $p=3$, one has 
$$\dim_k H^0(X\setminus E, S_X)/H^0(X, S_X)\geq 1.
$$
\end{proposition} 

\begin{proof}
A quadratic transformation $x'=x/y, y'=y, z'=z/y$ gives
equalities of derivations
$\partial/\partial x=(1/y')\partial /\partial x'$, 
$\partial /\partial y=\partial /\partial y'-(x'/y')\partial /\partial x'-
(z'/y')\partial /\partial z'$, 
$\partial /\partial z=(1/y')\partial /\partial z'$.
As was pointed out by Burns and Wahl \cite[Proposition~1.2]{Burns-Wahl}, 
any derivation $D$ of $(R, \mathfrak m)$ 
which satisfies $D(\mathfrak m)\subset \mathfrak m$ can be 
extended to a regular derivation on 
$\displaystyle \Proj \bigoplus_{i\geq 0} \mathfrak{m}^i$
(the point blow-up of $\mathfrak m$).
 
For $A_n:$ $R\cong k[[x, y, z]]/(xy+z^{n+1})$ with $p\,|\,(n+1)$,  
we have the exact sequence (cf. \cite[Theorem 25.2]{Matsumura86})
$$
0\to \left(\frac{\partial}{\partial z}, x\frac{\partial}{\partial x}-y\frac{\partial}{\partial y} \right)\to T_{\mathbf A_k^3}\otimes R
\stackrel{{\tiny( y\ x\ 0)}}{\longrightarrow}
\mathrm{Hom}_R (\mathcal I/\mathcal I^2,R ).
$$     
This is the Koszul complex associated with 
the regular sequence $x, y\in  R$.  
The derivation $\displaystyle \frac{\partial}{\partial z}\in \mathrm{Der}_k(R)$ 
does not lift
to $\displaystyle \Proj \bigoplus_{i\geq 0} \mathfrak{m}^i$
(the blow-up of $\mathfrak m$).

For $E_8^0:$ $R\cong k[[x, y, z]]/(z^2+x^3+y^5)$ in $p=5$,
one has the exact sequence 
$$
0\to \left(\frac{\partial}{\partial y}, 2z\frac{\partial}{\partial x}-3x^2\frac{\partial}{\partial z} \right)\to T_{\mathbf A_k^3}\otimes R
\stackrel{{\tiny( 3x^2\ 0\ 2z)}}{\longrightarrow} 
\mathrm{Hom}_R (\mathcal I/\mathcal I^2,R ). 
$$     
The derivation $\displaystyle \frac{\partial}{\partial y}\in \mathrm{Der}_k(R)$ 
does not lift
to the point blow-up $\Proj \bigoplus_{i\geq 0} \mathfrak{m}^i$.

For $E_7^0:$ $R\cong k[[x, y, z]]/(z^2+x^3+xy^3)$ in $p=3$, one has the 
exact sequence
$$
0\to \left(\frac{\partial}{\partial y}, z\frac{\partial}{\partial x}+y^3\frac{\partial}{\partial z} \right)\to T_{\mathbf A_k^3}\otimes R
\stackrel{{\tiny( y^3\ 0\ 2z)}}{\longrightarrow} 
\mathrm{Hom}_R (\mathcal I/\mathcal I^2,R ). 
$$     
The derivation $\displaystyle \frac{\partial}{\partial y}$ does not lift to the point blow-up.

For $E_6^0:$ $R\cong k[[x, y, z]]/(z^2+x^3+y^4)$ in $p=3$, one has
$$
0\to \left(\frac{\partial}{\partial x}, z\frac{\partial}{\partial y}+y^3\frac{\partial}{\partial z} \right)\to T_{\mathbf A_k^3}\otimes R
\stackrel{{\tiny( 0\ y^3\ 2z)}}{\longrightarrow} 
\mathrm{Hom}_R (\mathcal I/\mathcal I^2,R ). 
$$     
Two derivations $\displaystyle \frac{\partial}{\partial x}$, 
$\displaystyle y\frac{\partial}{\partial x}$ do not lift to 
the minimal resolution $X$.

For $E_8^0:$  $R\cong k[[x, y, z]]/(z^2+x^3+y^5)$ in $p=3$, one has
$$
0\to \left(\frac{\partial}{\partial x}, z\frac{\partial}{\partial y}-y^4\frac{\partial}{\partial z} \right)\to T_{\mathbf A_k^3}\otimes R
\stackrel{{\tiny( 0\ y^4\ z)}}{\longrightarrow} 
\mathrm{Hom}_R (\mathcal I/\mathcal I^2,R ). 
$$     
Two derivations $\displaystyle \frac{\partial}{\partial x}$, 
$\displaystyle y\frac{\partial}{\partial x}$ do not lift to 
the minimal resolution $X$.

For $E_8^1:$ $R\cong k[[x, y, z]]/(z^2+x^3+y^5+x^2y^3)$ in $p=3$, we have 
the derivation 
$D:=y\partial/\partial x -x\partial /\partial y\in \mathrm{Der}_k(R)$.
This satisfies $D(\mathfrak m )\subset \mathfrak m$, so this $D$ lifts to 
a point blow-up $\Proj \bigoplus_{i=0}^{\infty} \mathfrak m^i$,
on which lies a rational double point of type $E_7^0$.
But this $D$ does not lift to the minimal resolution, 
because it is the 
very element considered in $E_7^0$ above.

For $E_6^1$ in $p=3$, we have $R\cong k[[x, y, z]]/(z^2+x^3+y^4+x^2y^2)$.
The derivation 
$D:=(y-xy)\partial/\partial x+(x-y^2)\partial /\partial y+yz\partial/\partial z \in \mathrm{Der}_k(R)$
does not lift to the minimal resolution. 
\end{proof}

\section{Proof of the main theorem}\label{sec:4}

\begin{theorem}\label{theorem:main01}
Let $X\to \Spec R$ be the minimal resolution of a rational double point
defined over an algebraically closed field $k$ of characteristic $p\ne 2$.
Then the following assertions hold.
\begin{itemize}
\item [{\rm i)}] The natural morphism $H^1(S_X)\to 
H^1(X\setminus E, S_X)$
is an inclusion.
\item [{\rm ii)}] The dimension of $H_E^1(S_X\otimes \mathcal O_X(E))$
is zero except:
$$
\begin{cases}
1 \quad &\text{for}\quad E_8^0\text{ in }p=5$ and $E_6^0, E_7^0, E_8^1 \text{ in }p=3, \\
2  \quad &\text{for}\quad E_8^0\text{ in }p=3. \\
\end{cases}
$$ 
\item [{\rm iii)}]  One has an isomorphism 
$$H^0(X\setminus E, \Theta_{X})/H^0(X, \Theta_{X})\cong H_E^1(S_X),
$$
whose dimension is zero with the following exceptions:
$$
\begin{cases}
1   \quad & \text{for}\quad A_n  \text{ with }p\,|\,(n+1),\\
1  &\text{for}\quad E_8^0\text{ in }p=5 \text{ and } E_6^1, E_7^0, E_8^1\text{ in }p=3, \\
2  &\text{for}\quad E_6^0, E_8^0\text{ in }p=3. \\
\end{cases}
$$ 
\end{itemize}
\end{theorem} 

\begin{proof}
If the characteristic $p$ is a good prime (resp. very good prime) 
for the type of the rational 
double point $(\Spec R, \mathfrak m)$, 
Wahl's Theorem D (resp. Theorem C) provides 
the assertions i) and ii) (resp. iii)).
If $p$ is not a very good prime,    
one needs to show  that the lower estimates 
given in the previous section 
attain indeed the actual values.
This is trivially verified, since one has the inequality
coming from Proposition 2 and Theorem 1, 
$$
\#\{-2 \mbox{ curves in }E\}+h^1(S_X)+\dim_k H^0(X\setminus E, \Theta_{X})/H^0(X, \Theta_{X})\leq \tau_0.
$$
The Tjurina numbers in Artin's list~\cite{Artin77} say 
this is indeed an equality. 
\end{proof}

\begin{remark}
The pro-representable hull of equisingular deformations
of $X$ injects into a versal deformation of the rational
double point $(\Spec R, \mathfrak m)$. 
This forms a nonsingular subvariety whose dimension is 
as prescribed in  Theorem~\ref{theorem:main01}, ii).
One gets concrete one-parameter families given in the proof of 
Proposition~\ref{prop:equisingular}.
For $E_8^0\text{ in }p=3$, one has a two-parameter family: 
$z^2+x^3+y^5+s_1x^2y^3+s_2x^2y^2=0$
over $\Spec k[[s_1, s_2]]$.
Two strata of dimension one and zero respectively
can be observed in it. 
\end{remark}

\section{Characteristic $2$}\label{sec:5}

As is often the case with characteristic $2$, computation becomes 
more involved and demanding.
However, we can complete our evaluation of dimensions essentially 
in the same way as before.

\begin{theorem}\label{theorem:main02}
Let $X\to \Spec R$ be the minimal resolution of a rational double point
defined over an algebraically closed field $k$ of characteristic $2$.
Then we have the following assertions.
\begin{itemize}
\item [{\rm i)}] The natural morphism $H^1(S_X)\to 
H^1(X\setminus E, S_X)$
is an inclusion.
\item [{\rm ii)}] The dimension of $H_E^1(S_X\otimes \mathcal O_X(E))$
is zero except:
$$
\begin{cases}
1 \quad &\text{for}\quad E_6^0, E_7^2, E_8^3, \\
2 \quad &\text{for}\quad E_7^1, E_8^2, \\
3 \quad &\text{for}\quad E_7^0, E_8^1, \\
4 \quad &\text{for}\quad E_8^0, \\
n-1-r \quad &\text{for}\quad D_{2n}^r, D_{2n+1}^r.\\
\end{cases}
$$ 
\item [{\rm iii)}]  One has an isomorphism 
$$H^0(X\setminus E, \Theta_{X})/H^0(X, \Theta_{X})\cong H_E^1(S_X),
$$
whose dimension is zero with the following exceptions:
$$
\begin{cases}
1   \quad & \text{for}\quad A_n  \text{ with }2\,|\,(n+1),\\
1  \quad&\text{for}\quad E_6^0, E_7^3, E_8^3, \\
2  \quad&\text{for}\quad E_7^2, E_8^2, \\
3  \quad&\text{for}\quad E_7^1, E_8^1, \\
4  \quad&\text{for}\quad E_7^0, E_8^0, \\
n+1-r  \quad&\text{for}\quad D_{2n}^r, \\
n-r  \quad&\text{for}\quad D_{2n+1}^r. \\
\end{cases}
$$ 
\end{itemize}
\end{theorem} 

First we take care of derivations of rational double points of type $D_n$.
\begin{lemma}
For a rational double point of type $D$ in characteristic $2$,
the following derivations do not lift to the minimal resolution.
\begin{itemize}
\item[] $D_{2n}^0:$ $R\cong k[[x, y, z]]/(z^2+x^2y+xy^n)$. 
$$x\frac{\partial}{\partial z},\  \frac{\partial}{\partial z},\  
y\frac{\partial}{\partial z},\  y^2\frac{\partial}{\partial z}, \ \dots \ ,
\  y^{n-1}\frac{\partial}{\partial z}. $$
\item[] $D_{2n}^r:$ $R\cong k[[x, y, z]]/(z^2+x^2y+xy^n+xy^{n-r}z)$ 
with $r=1, 2, \dots, n-1$.
\begin{eqnarray*} &&nxy^{n-r-1}\frac{\partial}{\partial x}+y^{n-r}\frac{\partial}{\partial y}+(x+ry^{n-r-1}z)\frac{\partial}{\partial z},\  \mathfrak d_1,\  y\mathfrak d_1,\  y^2\mathfrak d_1, \ \dots\ , 
\ y^{n-r-1}\mathfrak d_1 \\
&&\qquad \text{ with }\  \mathfrak d_1:=x\frac{\partial}{\partial x}+\left(y^r+z\right)\frac{\partial}{\partial z}.
\end{eqnarray*}
\item[] $D_{2n+1}^0:$ $R\cong k[[x, y, z]]/(z^2+x^2y+y^nz)$.
$$\frac{\partial}{\partial x},\  
y\frac{\partial}{\partial x}, \ y^2\frac{\partial}{\partial x},
\ \dots \ ,\  y^{n-1}\frac{\partial}{\partial x}. $$
\item[] $D_{2n+1}^r:$ $R\cong k[[x, y, z]]/(z^2+x^2y+y^nz+xy^{n-r}z)$ with $r=1, 2, \dots, n-1$.
$$ \mathfrak d_2,\  y\mathfrak d_2,\  y^2\mathfrak d_2,\  \dots \ ,\  
y^{n-r-1}\mathfrak d_2
\quad \text{ with }\quad \mathfrak d_2:=\left(x+y^r\right)\frac{\partial}{\partial x}+z\frac{\partial}{\partial z}.
$$
\end{itemize}
\end{lemma}

\begin{proof}
Local calculation based on induction on $n\geq 2$. 
\end{proof}

\begin{proof}{\it\!of Theorem~\ref{theorem:main02} }
Each concrete one-parameter deformation $\mathcal X\to \Spec k[s]$ presented below
admits a simultaneous resolution without base extension,
providing a non-trivial equisingular deformation of the minimal resolution $X$.
   
For $D_{2n}^0$, one chooses an integer $k\in \{1, 2, \dots, n-1\}$.
The family is  $z^2+x^2y+xy^n+sxy^{n-k}z=0$, which has the singularity 
$D_{2n}^0$ on the special fiber $\mathcal X_0$,  and singularity $D_{2n}^k$ 
on a general fiber $\mathcal X_s$ $(s\ne 0)$.
For $D_{2n}^r$ wtih $1\leq r\leq n-2$, one chooses an integer $k\in \{r+1, r+2, \dots, n-1\}$.
The family is  $z^2+x^2y+xy^n+xy^{n-r}z+sxy^{n-k}z=0$, which has
$D_{2n}^r$ on the special fiber $\mathcal X_0$,  and $D_{2n}^k$ 
on a general fiber $\mathcal X_s$ $(s\ne 0)$.

For $D_{2n+1}^0$, one chooses an integer $k\in \{1, 2, \dots, n-1\}$.
The family is  $z^2+x^2y+y^nz+sxy^{n-k}z=0$, which has
$D_{2n+1}^0$ on the special fiber $\mathcal X_0$,  and $D_{2n+1}^k$ 
on a general fiber $\mathcal X_s$ $(s\ne 0)$.
For $D_{2n+1}^r$ wtih $1\leq r\leq n-2$, one chooses an integer $k\in \{r+1, r+2, \dots, n-1\}$.
The family is  $z^2+x^2y+y^nz+xy^{n-r}z+sxy^{n-k}z=0$, which has
$D_{2n+1}^r$ on $\mathcal X_0$,  and $D_{2n+1}^k$ on $\mathcal X_s$ $(s\ne 0)$.

For $E_6^0$, one chooses $z^2+x^3+y^2z+sxyz=0$, 
which has the singularity $E_6^0$ on $\mathcal X_0$,  and $E_6^1$ on $\mathcal X_s$ $(s\ne 0)$.

For $E_7^r$ with $r=0, 1, 2$, one chooses an integer $k\in \{r+1, r+2, \dots, 3\}$.
The deformation is given by $z^2+x^3+xy^3+\eta_r+s\eta_k=0$,
where $\eta_0:=0$, $\eta_1:=x^2yz$, $\eta_2:=y^3z$, $\eta_3:=xyz$.
This has the singularity $E_7^r$ on the special fiber $\mathcal X_0$,  and 
$E_7^k$ on $\mathcal X_s$ $(s\ne 0)$.

For $E_8^r$ with $r=0, 1, 2, 3$, one chooses an integer $k\in \{r+1, r+2, \dots, 4\}$.
The deformation is $z^2+x^3+y^5+\theta_r+s\theta_k=0$,
where $\theta_0:=0$, $\theta_1:=xy^3z$, $\theta_2:=xy^2z$, $\theta_3:=y^3z$, $\theta_4:=xyz$.
This has the singularity $E_8^r$ on the special fiber $\mathcal X_0$,  and 
$E_8^k$ on a general fiber $\mathcal X_s$ $(s\ne 0)$.

Hereafter, we give derivations of $(R, \mathfrak m)$
which do not lift to the minimal resolution.
For $A_n$, the derivation is exactly of the same form as before, so we omit it.

For $E_6^0:$ $R\cong k[[x, y, z]]/(z^2+x^3+y^2z)$, one has the exact sequence with
$\mathcal I=(z^2+x^3+y^2z)$,
$$
0\to \left(y^2\frac{\partial}{\partial x}+x^2\frac{\partial}{\partial z}, \frac{\partial}{\partial y}  \right)\to T_{\mathbf A_k^3}\otimes R
\stackrel{{\tiny( x^2\ 0\ y^2)}}{\longrightarrow} 
\mathrm{Hom}_R (\mathcal I/\mathcal I^2,R ). 
$$     
The derivation $\displaystyle \frac{\partial}{\partial y}\in \mathrm{Der}_k(R)$ 
does not lift to the point blow-up 
$\displaystyle\Proj \bigoplus_{i\geq 0} \mathfrak m^i$.

For $E_7^0:$ $R\cong k[[x, y, z]]/(z^2+x^3+xy^3)$,
one has the exact sequence with $\mathcal I=(z^2+x^3+xy^3)$,
$$
0\to \left(xy^2\frac{\partial}{\partial x}+(x^2+y^3)\frac{\partial}{\partial y}, 
\frac{\partial}{\partial z} \right)\to T_{\mathbf A_k^3}\otimes R
\stackrel{{\tiny( x^2+y^3\ xy^2\ 0)}}{\longrightarrow} 
\mathrm{Hom}_R (\mathcal I/\mathcal I^2,R ). 
$$     
Four derivations $\displaystyle \frac{\partial}{\partial z},
\  x\frac{\partial}{\partial z},\  y\frac{\partial}{\partial z}, 
\ y^2\frac{\partial}{\partial z}\in \mathrm{Der}_k(R)$ 
do not lift to the minimal resolution $X$.

For $E_8^0:$ $R\cong k[[x, y, z]]/(z^2+x^3+y^5)$,
one has the exact sequence 
$$
0\to \left(y^4\frac{\partial}{\partial x}+x^2\frac{\partial}{\partial y}, \frac{\partial}{\partial z} \right)\to T_{\mathbf A_k^3}\otimes R
\stackrel{{\tiny( x^2\ y^4\ 0)}}{\longrightarrow} 
\mathrm{Hom}_R (\mathcal I/\mathcal I^2,R ). 
$$     
Four derivations 
$\displaystyle \frac{\partial}{\partial z},\  y\frac{\partial}{\partial z},
\  y^2\frac{\partial}{\partial z}, 
\ x\frac{\partial}{\partial z}$ do not lift
to the minimal resolution $X$.

For $E_7^1:$ $R\cong k[[x, y, z]]/(z^2+x^3+xy^3+x^2yz)$.
The following three derivations do not lift to the minimal resolution.
$$ xy\frac{\partial}{\partial x}+y^2\frac{\partial}{\partial y}+(yz+x)\frac{\partial}{\partial z}, \ \mathfrak d_3, \ y\mathfrak d_3   \ \text{ with }
\mathfrak d_3=xz\frac{\partial}{\partial x}+(yz+x)\frac{\partial}{\partial y}+(z^2+y)\frac{\partial}{\partial z}.
$$ 

For $E_7^2:$ $R\cong k[[x, y, z]]/(z^2+x^3+xy^3+y^3z)$.
The following two derivations do not lift to the minimal resolution.
$$y\frac{\partial}{\partial y}+\left(x+z\right)\frac{\partial}{\partial z},\  
y\left(y\frac{\partial}{\partial y}+\left(x+z\right)\frac{\partial}{\partial z}\right)$$ 

For $E_7^3:$ $R\cong k[[x, y, z]]/(z^2+x^3+xy^3+xyz)$.
The following derivation does not lift to the minial resolution.
$$\displaystyle y\frac{\partial}{\partial y}+\left(y^2+z\right)\frac{\partial}{\partial z}$$ 

For $E_8^1:$ $R\cong k[[x, y, z]]/(z^2+x^3+y^5+xy^3z)$.
The following three derivations do not lift to the minimal resolution.
$$y^3\frac{\partial}{\partial x}+y^2z\frac{\partial}{\partial y}+\left(yz^2+x\right)\frac{\partial}{\partial z},\ 
\mathfrak d_4,\  y\mathfrak d_4
\text{ with } \mathfrak d_4:=y^2z\frac{\partial}{\partial x}+\left(yz^2+x\right)\frac{\partial}{\partial y}+\left(z^3+y\right)\frac{\partial}{\partial z}.
$$ 

For $E_8^2:$ $R\cong k[[x, y, z]]/(z^2+x^3+y^5+xy^2z)$.
The following two derivations do not lift to the minimal resolution.
$$ x\frac{\partial}{\partial y}+y^2\frac{\partial}{\partial z},\ 
y^2\frac{\partial}{\partial x}+ z\frac{\partial}{\partial y}+x\frac{\partial}{\partial z}
$$ 

For $E_8^3:$ $R\cong k[[x, y, z]]/(z^2+x^3+y^5+y^3z)$.
The following derivation does not lift to the minimal resolution.
$$\displaystyle y\frac{\partial}{\partial y}+\left(y^2+z\right)\frac{\partial}{\partial z}$$ 

We combine these with the previous lemma, 
inequalities in Preliminaries and
the Tjurina numbers in Artin's list~\cite{Artin77} 
to get the assertions i), ii), iii). 
\end{proof}

\specialsection*{Acknowledgement}
I would like to express my sincere gratitude to 
Professors Kei-ichi~Watanabe, Shihoko~Ishii,
Tadashi~Tomaru, Masataka~Tomari for valuable suggestions,
inspirations and
especially for having encouraged me to participate in the seminar on 
singularities since I was a student in Master's course.

\bibliographystyle{amsplain}

\end{document}